\documentclass[11pt]{article}
\usepackage[latin1]{inputenc}
\usepackage{amsmath,amsthm,amssymb}
\usepackage{amsfonts}
\usepackage{amsmath,amsthm,amssymb,amscd}
\usepackage{latexsym}
\usepackage{color}
\usepackage{graphicx}
\usepackage{enumitem}
\usepackage{mathrsfs}
\usepackage{mathtools}

\makeatletter \@addtoreset{equation}{section} \makeatother

\newtheorem{theorem}{Theorem}[section]
\newtheorem{definition}{Definition}[section]

\newtheorem{lemma}{Lemma}[section]
\newtheorem{remark}{Remark}[section]

\begin{document}		

	\title{\bf \sc Mixed order elliptic problems driven by a singularity, a Choquard type term and a discontinuous power nonlinearity with critical variable exponents}
	\author{\bf Jiabin Zuo$^{a}$ \&  Debajyoti Choudhuri$^b$ \& Du\v{s}an D.  Repov\v{s}$^{c,d,e}$ \footnote{Corresponding
			author: dusan.repovs@guest.arnes.si}\\
		\small{$^a$School of Mathematics and Information
			Science, Guangzhou University, Guangzhou, 510006, China.}\\
		\small{$^b$Department of Mathematics, National Institute of Technology Rourkela, Rourkela - 769008, Odisha, India.}\\
		\small{$^c$Faculty of Education, University of Ljubljana, Kardeljeva pl. 16, SI-1000 Ljubljana, Slovenia.}\\
		\small{$^d$Faculty of Mathematics and Physics, University of Ljubljana, Jadranska 21, SI-1000 Ljubljana, Slovenia.}\\
		\small{$^e$Institute of Mathematics, Physics and Mechanics,  Jadranska 19, SI-1000 Ljubljana, Slovenia.}\\
		\small{\it Email: zuojiabin88@163.com \& dc.iit12@gmail.com \& dusan.repovs@guest.arnes.si}}
	\date{}
	\maketitle
	\begin{abstract}
We prove the existence of solutions for the following critical Choquard type problem with a variable-order fractional Laplacian and a variable singular exponent
\begin{align*}
	\begin{split}
		a(-\Delta)^{s(\cdot)}u+b(-\Delta)u&=\lambda |u|^{-\gamma(x)-1}u+\left(\int_{\Omega}\frac{F(y,u(y))}{|x-y|^{\mu(x,y)}}dy\right)f(x,u)\\
		&
		+\eta H(u-\alpha)|u|^{r(x)-2}u,~\text{in}~\Omega,\\
		u&=0,~\text{in}~\mathbb{R}^N\setminus\Omega.
	\end{split}
\end{align*}
where $a(-\Delta)^{s(\cdot)}+b(-\Delta)$ is a mixed operator with variable order $s(\cdot):\mathbb{R}^{2N}\rightarrow (0,1)$, $a, b\geq 0$ with $a+b>0$,  $H$ is the Heaviside function  (i.e., $H(t)=0$ if $t\leq0$, $H(t) = 1$ if $t>0),$ $\Omega\subset\mathbb{R}^N$ is a bounded domain, $N\geq 2$, $\lambda>0$, $0<\gamma^{-}=\underset{x\in\bar{\Omega}}{\inf}\{\gamma(x)\}\leq\gamma(x)\leq\gamma^+=\underset{x\in\bar{\Omega}}{\sup}\{\gamma(x)\}<1$, $\mu$ is a continuous variable parameter, and $F$ is the primitive function of a suitable $f$. The variable exponent $r(x)$ can be equal to the critical exponent $2_{s}^*(x)=\frac{2N}{N-2\bar{s}(x)}$ with $\bar{s}(x)=s(x,x)$ for some $x\in\bar{\Omega},$ and $\eta$ is a positive parameter. We also show that as $\alpha\rightarrow 0^+$, the corresponding solution converges to a solution for the above problem with $\alpha=0$.

	\begin{flushleft}
 {\bf Math. Subj. Classif.}:~35R11, 35J75, 35J60, 46E35.
\end{flushleft}
\end{abstract}

	\section{Introduction}\label{Introduction}	

In this paper, we shall consider the following  mixed order elliptic problem
\begin{align}
	\begin{split}\label{main prob}
		a(-\Delta)^{s(\cdot)}u+b(-\Delta)u&=\lambda |u|^{-\gamma(x)-1}u+\left(\int_{\Omega}\frac{F(y,u(y))}{|x-y|^{\mu(x,y)}}dy\right)f(x,u)\\
		&
		+\eta H(u-\alpha)|u|^{r(x)-2}u,~\text{in}~\Omega,\\
		u&=0,~\text{in}~\mathbb{R}^N\setminus\Omega.
	\end{split}
\end{align}
Here, $(-\Delta)^{s(\cdot)}$ is the fractional Laplacian of a variable order $s(\cdot)$ which is defined as follows.
\begin{eqnarray}
	(-\Delta)^{s(\cdot)}u(x)&=&C_{N,s(\cdot)}\lim_{\epsilon\rightarrow 0}\int_{\mathbb{R}^N\setminus B_{\epsilon}(x)}\frac{(u(x)-u(y))}{|x-y|^{N+2s(x,y)}}dy, x\in\mathbb{R}^N\nonumber
\end{eqnarray}
and $-\Delta=-\sum_{i=1}^{N}\frac{\partial^2}{\partial x_i^2}$ is the well-known Laplace operator. 

We suppose that $s(\cdot):\mathbb{R}^{2N}\rightarrow(0,1)$ is a uniformly continuous function and $\mu(\cdot):\mathbb{R}^{2N}\rightarrow(0,N)$ is a continuous function that satisfy the following
hypotheses
\begin{eqnarray}\label{H_cond}
	(H_1):~& &0<s^-\leq s^+<1, 0<\mu^-<\mu^+<N~\text{and}~2s^+<N\nonumber;\\
	(H_2):~& &s(\cdot), \mu(\cdot)~\text{are symmetric, i.e.}\nonumber\\
~& &s(x,y)=s(y,x), \mu(x,y)=\mu(y,x), ~(x,y)\in\mathbb{R}^{2N}.\nonumber
\end{eqnarray}

We assume that $f:\bar{\Omega}\times\mathbb{R}\rightarrow\mathbb{R}^+\cup\{0\}$ is a continuous  odd function with respect to the second variable that satisfies the following hypotheses
\begin{eqnarray}\label{F_cond}
	(F_1):~& &\text{there exist}~c_1>0~\text{and}~\tau\in C_+(\bar{\Omega})\cap\mathcal{M}~\text{with}~2<\tau^-<{2_s^*}^-\nonumber\\
~& &~\text{and}~2<\tau^+<{2_s^*}^+~\text{such that}~0\leq f(x,t)\leq c_1|t|^{\tau(x)-1},~\nonumber\\ ~& &\text{for any}~(x,t)\in\Omega\times\mathbb{R},\text{where}~\nonumber\\ ~& &~\mathcal{M}=\{\tau\in C_+(\bar{\Omega}):2\leq\tau(x)\beta^-\leq\tau(x)\beta^+<{2}_s^*(x)~\text{for all}~x\in\mathbb{R}^N\}\nonumber\\ ~& &~\text{with}~\beta\in C_+(\mathbb{R}^{2N})~\text{such that}~\frac{2}{\beta(x,y)}+\frac{\mu(x,y)}{N}=2~\nonumber\\ ~& &~\text{for any}~(x,y)\in\mathbb{R}^{2N},\text{and}\nonumber\\ ~& &~~1<{2_s^*}^{-}=\underset{x\in\bar{\Omega}}{\inf}\{{2_s^*}(x)\}\leq {2_s^*}(x)\leq {2_s^*}^+=\underset{x\in\bar{\Omega}}{\sup}\{{2_s^*}(x)\};\nonumber\\
	(F_2):~& &~\text{there exists}~\theta>2~\text{such that}~\nonumber\\
	& &0<\theta F(x,t)\leq 2tf(x,t),~\text{for any}~t\in\mathbb{R}\setminus\{0\}~\text{and for any}~x\in\Omega\nonumber.
\end{eqnarray}
\begin{remark}\label{rema1}
	We shall  sometimes denote the {\it Choquard} type of term $$\int_{\Omega}\frac{F(y,u(y))}{|x-y|^{\mu(x,y)}}dy$$ by $F*\frac{1}{|x|^{\mu(\cdot)}}$.
\end{remark}
 In recent years, the problem of combining local and nonlocal Laplacian operators has attracted a huge amount of attention due to theoretical and practical applications, such as in plasma physics (see \cite{d12}) and population dynamics (see \cite{d23}). Another key reason as to why the study of such equations holds importance is due to the fact that many of the free boundary problems studied in mathematical physics are formulated in this form. Some of them are the surface problem, the obstacle problem, and the Elenbaas equation (see 
 \cite{16,17,Chang 3}). Interested readers may also refer to \cite{5}.

Biagi et al. \cite{2022add} considered the corresponding Sobolev inequality, detecting
the optimal constant for a critical local and nonlocal problem, furthermore, they proved an existence (and nonexistence) result of the corresponding subcritical perturbation problem. Hu and Yang \cite{bh} investigated a mixed order problem and proved a bifurcation and multiplicity result. Arora and R\u{a}dulescu \cite{arora1} studied an elliptic problem involving both local-nonlocal operators and proved the existence (or nonexistence) of weak solution to it. Cassani et al. \cite{DC} analyzed spectral properties, established the validity of the maximum principle, showed related properties of weak solutions of mixed local-nonlocal elliptic equations. In particular, this makes the problem more interesting when the order of nonlocal operator can be variable, and to the best of our knowledge the first fractional variable-order problem can be traced back to the work of Lorenzo and Hartley \cite{vo}. Since the fractional problem of variable order can better describe the temperature change of the object in the process of nonlinear diffusion, we shall  focus on a class of problems with the fractional variable-order Laplace operator (see \cite{lz,mdby,nuovo}).

The starting point for the study of singular problems was the pioneering work of Crandall \cite{MGC}. We also refer to the monograph by Ghergu and R\u{a}dulescu \cite{MGV}. At a time when scholars were concerned about nonlinearities with constant exponent, few works also involved variable singular exponent case. For instance, Garain et al. \cite{GSC} proved the existence of one solution and two solutions for  a quasilinear variable singular exponent fractional elliptic problems when the parameters are in a suitable range. Ghosh et al. \cite{SGD} proved the existence of infinite number of small energy solutions by using the symmetric mountain pass lemma for the elliptic PDEs of variable exponent with a singular nonlinearity. Interested readers can also refer to Chammem et al. \cite{crg}.

 The Choquard type problem is a hot topic which many mathematical scholars have begun paying close attention to recently. It becomes more interesting when combined with the variable exponents. Alves et al. \cite{alv1} established   a kind of Hardy-Littlewood-Sobolev inequality with variable exponents, which usually is used to estimate the convolution term in a problem. With the help of Nehari manifold methods, Li and Tang \cite{gdl} obtained the existence of positive ground state solution for a class of stationary Choquard equations with variable exponent growth. By using variational methods, Zuo et al. \cite{ZJF} and Biswas et al. \cite{RBS} proved the existence and multiplicity of solutions for a ($p(.)\&q(.)$) Choquard problem with nonstandard growth.

 In practical problems, some functional models are not always continuous, and it is necessary to consider non-smooth functions on a bounded domain, Figueiredo et al. \cite{gmfrg} studied a critical $p\&q$ elliptic problem with  discontinuous nonlinearities and obtained  the existence of positive solutions. Yuan-Yu \cite{ZYJX}  also considered a class of discontinuous problems to which they obtained the existence of definite sign and sign-changing solutions. Moreover, Xiang-Zhang \cite{MXBZ} dealt with a class of critical $p$-Kirchhoff type problems with a discontinuous nonlinearity in the whole space $\mathbb{R}^{N}$ and proved the existence and multiplicity of solutions. Albuquerque et al. \cite{jcda} discussed a a coupled elliptic system with a discontinuous nonlinearity and showed the existence and behavior of positive solutions.  For more information on non-smooth analysis, one can also see Alves-Yuan \cite{COA}, dos Santos-Tavares \cite{tava1}, dos Santos-Figueiredo \cite{ggdos}, Saoudi et al. \cite{KAD}.

 Inspired by all of the above, we are going to consider a critical problem involving  the mixed operator,  variable singularity exponents, the convolution term and the discontinuity term simultaneously. To this end, we need to overcome the difficulty posed by each part.
Here are the main results of this article.

\begin{theorem}\label{mainres1}
	Let $(H_1)-(H_2)$ and $(F_1)-(F_2)$ hold for  functions $s$, $\mu$ and $f.$ Then there exist $\alpha', \lambda'>0$ such that for every $\alpha\in (0,\alpha')$, $\lambda\in(0,\lambda')$, problem \eqref{main prob} admits at least one nontrivial weak solution, say $u_{\alpha}$. Furthermore, $|\{x\in\Omega:u_{\alpha}(x)>\alpha\}|>0$, where $|\cdot|,$ when considered for a set, denotes the Lebesgue measure.
\end{theorem}

\begin{theorem}\label{mainres2}
	Let $u_{\alpha}$ be a weak solution to
	problem
	\eqref{main prob}. Then for any sequence $(\alpha_i)\subset(0,\alpha_0)$ with $\alpha_i\rightarrow 0^+$, there exists a subsequence such that $u_{\alpha_i}\rightarrow u_0$ in $Z$, where $u_0$ is a nontrivial weak solution to the following problem
	\begin{align*}
		\begin{split}
			a(-\Delta)^{s(\cdot)}u+b(-\Delta)u=&\lambda |u|^{-\gamma(x)-1}u+\left(\int_{\Omega}\frac{F(y,u(y))}{|x-y|^{\mu(x,y)}}dy\right)f(x,u)\\
			&+\eta H(u)|u|^{r(x)-2}u,~\text{in}~\Omega,\\
			u=&0,~\text{in}~\mathbb{R}^N\setminus\Omega.
		\end{split}
	\end{align*}
	which we shall denote by $(P_0)$.
\end{theorem}

The paper is organized as follows. In Section \ref{functional setting}, we introduce some basic workspace and properties and crucial technical lemmas. In Section \ref{auxresults}, we construct cut-off functions, auxiliary results, and verify all conditions of the non-smooth critical point theory. In Section \ref{proof}, we prove the main results. In Section \ref{appendix}, we explain some more details.
\section{Preliminaries}\label{functional setting}
 Let $X$ be a space which is defined as follows
$$X=\big\{u:\mathbb{R}^N\rightarrow\mathbb{R}:~u
\
\hbox{  is  measurable,}
\
u|_{\Omega}\in L^2(\Omega),
\
 \frac{u(x)-u(y)}{|x-y|^{\frac{N+2s(x,y)}{2}}}\in L^2(Q)\big\}$$
and is equipped with the Gagliardo norm
$$\|u\|_{X}=\|u\|_{2}+\left(\iint_{Q}\frac{|u(x)-u(y)|^2}{|x-y|^{N+2s(x,y)}}dydx\right)^{\frac{1}{2}},$$
where $$\Omega\subset\mathbb{R}^N,
\
Q=\mathbb{R}^{2N}\setminus((\mathbb{R}^N\setminus\Omega)\times(\mathbb{R}^N\setminus\Omega)).$$
 Here, $\|u\|_2$ refers to the $L^2$-norm of $u$. 
 
We shall mostly use the subspace $X_0$ of $X$ which is defined as
$$X_0=\left\{u\in X:u=0~\text{in}~\mathbb{R}^N\setminus\Omega\right\}$$
with the norm
$$\|u\|_{s(\cdot,\cdot),2}=\left(\iint_{Q}\frac{|u(x)-u(y)|^{2}}{|x-y|^{N+2s(x,y)}}dydx\right)^{\frac{1}{2}}.$$
Then $(X_0,\|\cdot\|_{s(\cdot,\cdot),2})$ is a Hilbert space with respect to the norm $\|\cdot\|_{s(\cdot,\cdot),2}$ induced by the inner product $$
\langle u,v\rangle=\iint_{Q}\frac{(u(x)-u(y))(v(x)-v(y))}{|x-y|^{N+2s(x,y)}}dxdy$$ for each $u,v\in X_0$. The space $X_{0}$ is also separable and reflexive, see Azroul et al.~\cite[Lemma 2.3]{m4}.

 In order to study  \eqref{main prob}, we  define a new workspace
$$Z=\left\{u\in L^2
(\Omega):u=0~\text{a.e. in}~\mathbb{R}^N\setminus\Omega, a\|u\|_{s(\cdot,\cdot),2}^2+b\|u\|_{1,2}^2<\infty\right\}$$
with the norm:
\begin{eqnarray}\label{norm}
	\|u\|&=&(a\|u\|_{s(\cdot,\cdot),2}^2+b\|u\|_{1,2}^2)^{\frac{1}{2}},
\end{eqnarray}
where $$\|u\|_{1,2}^2=\int_{\Omega}|\nabla u|^2dx.$$

A standard approach from functional analysis guarantees the reflexivity of  Banach space $Z$.
\begin{remark}\label{space_equiv}	We make some useful observations.
	\begin{enumerate}
		\item The spaces
		$X,  X_0\subset W^{s(x),2}(\Omega)$,   where
		$$W^{s(x),2}(\Omega)=\left\{u\in L^2(\Omega):\frac{u(x)-u(y)}{|x-y|^{\frac{N+2s(x,y)}{2}}}\in L^2(\Omega)\right\}$$
		is the usual fractional order
		Sobolev space, endowed with the Gagliardo norm
		 $$\|u\|_{W^{s(\cdot),2}(\Omega)}=\|u\|_{2}+\left(\iint_{\Omega\times\Omega}\frac{|u(x)-u(y)|^2}{|x-y|^{N+2s(x,y)}}dydx\right)^{\frac{1}{2}}.$$
		\item  Due to the ``boundary condition'' that $u=0$ in $\mathbb{R}^N\setminus\Omega$, the space $X_0$ is equivalent to $W_0^{s(x),2}(\mathbb{R}^N)$.
		\item The first order Sobolev space will be defined as
		$$W_0^{1,2}(\Omega)=\left\{u\in L^2(\Omega):u|_{\partial\Omega}=0, \int_{\Omega}|\nabla u|^2dx<\infty\right\}.$$
		\item Any uniform constant appearing in an inequality will be denoted by the symbol $C$ with or without a prefix/suffix.
	\end{enumerate}
\end{remark}
 Recall that the Lebesgue space of variable exponent is defined as follows (see Papageorgiou et al. \cite{RD}).
\begin{align}\label{lpx}
	L^{p(x)}(\Omega)&=\left\{u:\Omega\rightarrow\mathbb{R}:~u~\text{is measurable and}~\int_{\Omega}|u(x)|^{p(x)}dx<\infty\right\}
\end{align}
and it is
equipped with the  Luxemburg norm
$$\|u\|_{p(\cdot)}=\inf\left\{\mu>0:\int_{\Omega}\left|\frac{u(x)}{\mu}\right|^{p(x)}dx<1\right\}.$$

Furthermore, we define the best Sobolev constants as follows:
\begin{align}\label{sobolev_constant}
	S_1&=\underset{u\in X_0\setminus\{0\}}{\inf}\frac{\iint_{Q}\frac{|u(x)-u(y)|^{2}}{|x-y|^{N+2s(x,y)}}dydx}{\|u\|_{r(\cdot)}^2},
\end{align}
\begin{align}\label{sobolev_constant_2}
	S_2&=\underset{u\in W_0^{1,2}(\Omega)\setminus\{0\}}{\inf}\frac{\int_{\Omega}|\nabla u|^2dx}{\|u\|_{2^*}^2},
\end{align}
where $2^*=\frac{2N}{N-2}$.
An important inequality which will be used consistently throughout the paper is the H\"{o}lder inequality for the Lebesgue space of variable exponent, see Fan-Zhao~\cite[Theorem 1.3]{Fan2001}.

 Finally, the modular function is defined to be the mapping $\rho : L^{q(x)}(\Omega)\rightarrow\mathbb{R}$ which is defined as
\begin{align*}
	\rho_{q(x)}(u):= \int_{\Omega}|u|^{q(x)}dx.
\end{align*}
For some important properties of this modular function, we refer the reader to \cite {RD}.

\section{Auxiliary results}\label{auxresults}
 We define
$$J_{\alpha}(t)=\int_{0}^{t}j_{\alpha}(s)ds=\frac{1}{r(x)}H(t-\alpha)(t^{r(x)}-\alpha^{r(x)})$$
to be the primitive of
$$j_{\alpha}(t)=H(t-\alpha)|t|^{r(x)-2}t, 
g(x,t)=\left(F*\frac{1}{|x|^{\mu(\cdot)}}\right)f(x,t), 
F(x,t)=\int_{0}^{t}f(x,s)ds.$$

We call $u\in Z$  a {\it weak solution} to problem $\eqref{main prob}$ if for every $\varphi\in Z$, we have
\begin{align}
	&a\iint_{Q}\frac{(u(x)-u(y))}{|x-y|^{N+2s(x,y)}}(\varphi(x)-\varphi(y)) dx dy\nonumber+b\int_{\Omega\time\Omega}\nabla u\cdot\nabla\varphi dx\\
	&=\lambda\int_{\Omega}|u|^{-\gamma(x)-1}u\varphi dx+\int_{\Omega}g(x,u)\varphi dx+\eta\int_{\Omega}j_{\alpha}(u)\varphi dx.
\end{align}
Clearly, the weak solution for equation $\eqref{main prob}$ is a {\it critical point} of the following functional
\begin{align}
	\begin{split}\label{energy fnal}
		I(u)=&\frac{a}{2}\iint_{Q}\frac{|u(x)-u(y)|^{2}}{|x-y|^{N+2s(x,y)}}dxdy+\frac{b}{2}\int_{\Omega}|\nabla u|^2dx-\lambda\int_{\Omega}\frac{|u|^{1-\gamma(x)}}{1-\gamma(x)}dx\\
		&-\frac{1}{2}\int_{\Omega}\left(F*\frac{1}{|x|^{\mu(\cdot)}}\right)F(x,u)dx-\eta\int_{\Omega}J_{\alpha}(u)dx\\
		 =&\frac{1}{2}\|u\|^2-\lambda\int_{\Omega}\frac{|u|^{1-\gamma(x)}}{1-\gamma(x)}dx-\frac{1}{2}\int_{\Omega}\left(F*\frac{1}{|x|^{\mu(\cdot)}}\right)F(x,u)dx\\
&-\eta\int_{\Omega}J_{\alpha}(u)dx.
	\end{split}
\end{align}
However, the functional $I$ is not even differentiable, owing to the existence of the singular term and discontinuous nonlinear term. Therefore, we replace the original functional $I$ with
the cut-off functional $\bar{I}$.
\subsection{Cut-off functional}\label{existence}
 To overcome the nondifferentiability of the functional $I$ due to the singular term,  we define the following cut-off function
$$\phi(x,t)=\begin{cases}
	|t|^{-\gamma(x)-1}t,& ~\text{if}~ |t|>\underline{u}_{\lambda}(x)\\
	\underline{u}_{\lambda}(x)^{-\gamma(x)},&~\text{if}~ |t|\leq\underline{u}_{\lambda}(x).
\end{cases}$$
Here, $\underline{u}_{\lambda}$ refers to the solution of the singular problem as in Theorem \ref{imp_thm} (see Section~\ref{appendix}). We get the well-defined {\it cut-off} problem:
\begin{align}\label{main3}
	\begin{split}
		a(-\Delta)^{s(\cdot)}u-b\Delta u&= \lambda \phi(x,t)+\left(F*\frac{1}{|x|^{\mu(\cdot)}}\right)f(x,u)\\&+\eta H(u-\alpha)|u|^{r(x)-2}u~\text{in}~\Omega,\\
		u&=0~\text{in}~\mathbb{R}^N\setminus\Omega.
	\end{split}
\end{align}
We say that
$\tilde{u}\in Z$ is a {\it weak solution} for equation $\eqref{main3}$ if we have
\begin{align*}
	&a\iint_{Q}\frac{(\tilde{u}(x)-\tilde{u}(y))}{|x-y|^{N+2s(x,y)}}(\varphi(x)-\varphi(y)) dx dy\nonumber+b\int_{\Omega\time\Omega}\nabla \tilde{u}\cdot\nabla\varphi dx\\
	&=\lambda\int_{\Omega}\phi(x,\tilde{u})\varphi dx
	+\int_{\Omega}\left(\int_{\Omega}\frac{F(y,\tilde{u}(y))}{|x-y|^{\mu(x,y)}}dy\right)f(x,\tilde{u})\varphi dx
	+\eta\int_{\Omega}j_{\alpha}(\tilde{u})\varphi dx
\end{align*}
for any $\varphi\in Z$.
Consequently, we shall  define the energy functional
\begin{align}
	\begin{split}\label{energy3}
		\bar{I}(u)=&\frac{a}{2}\iint_{Q}\frac{|u(x)-u(y)|^{2}}{|x-y|^{N+2s(x,y)}}dxdy+\frac{b}{2}\int_{\Omega}|\nabla u|^2dx-\lambda\int_{\Omega}\Phi(x,u)dx\\
		&-\frac{1}{2}\int_{\Omega}\left(\int_{\Omega}\frac{F(y,u(y))}{|x-y|^{\mu(x,y)}}dy\right)F(x,u) dx-\eta\int_{\Omega}J_{\alpha}(u)dx,
	\end{split}
\end{align}
where $\Phi$ is the primitive of the function $\phi$, i.e. $\Phi(x,t)=\int_0^t\phi(x,s)ds$.
\begin{remark}\label{positivity}
	If $u_0$ is a global minimizer of $I$, then clearly $I(u_0)=I(|u_0|)$. Trivially, we have $I(u_0)\leq I(|u_0|)$. Furthermore, using
	$$\iint_{Q}\frac{||u(x)|-|u(y)||^{2}}{|x-y|^{N+2s(x,y)}}dxdy\leq \iint_{Q}\frac{|u(x)-u(y)|^{2}}{|x-y|^{N+2s(x,y)}}dxdy,$$
	$$-\int_{\Omega}H(|u|-\alpha)|u|^{r(x)}dx\leq-\int_{\Omega}H(u-\alpha)|u|^{r(x)}dx,$$
	and the property of $F$ that it is even, we have $I(u_0)\geq I(|u_0|)$.
	Hence, we seek for a minimizer of the functional $I$ which is nonnegative. Furthermore, the solution is positive a.e. in $\Omega$ due to singular terms.
\end{remark}
\begin{remark}\label{key_obs} We note that  the weak solution of new problem \eqref{main3} with $u>\underline{u}_{\lambda}$ is also a weak solution to original problem \eqref{main prob}. Refer to the proof of Theorem \ref{imp_thm} in Section~\ref{appendix}.
\end{remark}
\subsection{Definitions from convex analysis}
 Let us now see a few definitions that are very important in the analysis of functionals that are nondifferentiable but {\it locally Lipschitz continuous} (see \cite[Section $2$]{tava1}).
\begin{definition}(see  \cite[Section $2$]{tava1})
	Let $J$ be a locally Lipschitz continuous functional. The {\it directional derivative} of $J$ at $u$ in the direction of $z\in V$ is defined by
	$$\tilde{J}(u;z)=\lim\limits_{h\rightarrow0}\sup_{\xi\downarrow0}\frac{J(u+h+\xi z)-J(u+h)}{\xi}.$$
\end{definition}
 From dos Santos-Tavares \cite{tava1} we know that $\tilde{J}(u;\cdot)$ is convex and continuous. The subdifferential of $\tilde{J}(u;\cdot)$ at $w\in V$ is defined to be the following set:
$$\partial\tilde{J}(u;w)=\{\upsilon\in V^*:\tilde{J}(u;z)\geq\tilde{J}(u;w)+\langle \upsilon, z-w\rangle, ~z\in V\}.$$
Here, $\langle\cdot,\cdot\rangle$ denotes the duality pair between $V$ and $V^*$ (the dual space of $V$). The {\it generalized gradient}
(see \cite{tava1}) of $J$ at $u$ is defined as
$$\partial J(u)=\{\upsilon\in V^*:\langle \upsilon, z\rangle\leq \tilde{J}(u;z), ~z\in V\},$$
which is convex and weak$^*$- compact by \cite{Chang 3}.
Clearly, $\partial J(u)$ is nonempty and is the subdifferential of $\tilde{J}(u;0)$ as $\tilde{J}(u;0)=0$.
{Then by} \cite{Chang 3}, {the function}
\begin{equation}\label{sub}
	\text{{$\Lambda_J(u)=\min\{\|\upsilon\|_{V^*}:\upsilon \in \partial J(u)\}$}}
\end{equation}
{exists and is lower semi-continuous.} 

A function $\bar{u}$ is said to be a {\it critical point} of $J$ if $0\in \partial J(\bar{u})$, whereas $c\in
\mathbb{R}$ is said to be a {\it critical value} of $J$ if $J(\bar{u})=c$ for a critical point $\bar{u}\in V$.
If $J$ is a $C^1$ {functional} then $\partial J(u)=\{J^\prime(u)\}$.

\begin{remark}\label{SUBDIFF}
	Define 
	\begin{align*}\bar{I}_{\eta}(u)=&\int_{\Omega}J_{\alpha}(u)dx
	\end{align*} 
	and
\begin{align*} 
	E_{0}(u)=&\frac{\|u\|^2}{2}-\int_{\Omega}\Phi(x,u)dx-\int_{\Omega}\left(F*\frac{1}{|x|^{\mu(\cdot)}}\right)F(x,u)dx.
	\end{align*} Then
	\begin{eqnarray}\label{3.9}\partial\bar{I}(w)\subset \{E'_{0}(w)\}-\partial\bar{I}_{\eta}(w)~\text{for all}~w\in Z.\end{eqnarray}
\end{remark}
\begin{remark}\label{positive_soln_reason}
	We shall  be interested in solutions that are greater than $\alpha>0$.
	We shall  show that $\{u>\alpha\}$ is of positive Lebesgue measure for a small range of $\alpha$.

	Next, consider the energy functional
	 	\begin{align*}
\mathcal{I}(u)=&\frac{1}{2}\|u\|^2-\lambda\int_{\Omega}\frac{(u^+)^{1-\gamma(x)}}{1-\gamma(x)}dx-\frac{1}{2}\int_{\Omega}\left(F*\frac{1}{|x|^{\mu(\cdot)}}\right)F(x,u^+)dx\\&-\eta\int_{\Omega}J_{\alpha}(u^+)dx.	 \end{align*}
	A critical point of $\mathcal{I}$ is clearly a weak solution of \eqref{main prob}. Therefore we have
	\begin{align}0=&a\iint_{Q}\frac{(u(x)-u(y))}{|x-y|^{N+2s(x,y)}}(\varphi(x)-\varphi(y)) dx dy\nonumber+b\int_{\Omega\time\Omega}\nabla u\cdot\nabla\varphi dx-\nonumber\\
		&\lambda\int_{\Omega}(u^+)^{-\gamma(x)}\varphi dx-\frac{1}{2}\int_{\Omega}\left(F*\frac{1}{|x|^{\mu(\cdot)}}\right)f(x,u^+)\varphi dx-\eta\int_{\Omega}j_{\alpha}(u^+)\varphi dx.\nonumber
	\end{align}
	On testing with $\varphi=u^-$, we get that the Lebesgue measure of the set $\{u<0\}$ is zero. Thus without loss of generality we may assume that $u>0$ a.e. in $\Omega$.
\end{remark}
 We recall some properties of the functionals $\bar{I}_\eta$ in the following lemma, the proof of which follows from the argument in dos Santos-Figueiredo \cite[Lemma $3.1$]{21}.
\begin{lemma}\label{properties I}
	$\bar{I}_\eta:L^{r(x)}(\Omega)\rightarrow\mathbb{R}$ is a locally Lipschitz functional
	such that $\partial \bar{I}_\eta(u)\subset [\underline{f}_\alpha(u(x)),\overline{f}_\alpha(u(x))]$ a.e. in $\Omega$ and
	\begin{eqnarray}
		\begin{split}
			[\underline{f}_\alpha(u),\overline{f}_\alpha(u)]=&\begin{cases}
				\{0\}, & ~\text{if}~u<\alpha\\
				[0,u^{r(x)-1}], &~\text{if}~u=\alpha\\
				\{u^{r(x)-1}\}, &~\text{if}~u>\alpha.
			\end{cases}
		\end{split}
	\end{eqnarray}
\end{lemma}
\subsection{The Mountain Pass Theorem}
 An essential result which will be required henceforth is the Mountain Pass Theorem for a locally Lipschitz nondifferentiable functional, due to { Grossinho-Tersian} \cite{26}, and { R\u{a}dulescu} \cite{32}.
 The following theorem gives an energy level, below which the $(PS)_c$ condition is satisfied by the functional $\bar{I}$.
\begin{theorem}\label{PS_THM}
	The functional $\bar{I}$ satisfies the $(PS)_c$-condition for\\
	$$c<
	 \left(\frac{1}{2}-\frac{1}{\kappa}\right){S_1}^{\frac{\mathcal{V}}{\mathcal{V}-2}}-\left(\frac{1}{2}-\frac{1}{\kappa}\right)^{-\frac{1-\delta}{1+\delta}}\left(\frac{\lambda}{1-\gamma^+}C(\Omega,N)\right)^{\frac{2}{1+\delta}}=c_*$$ where $c_*>0$ for any $\lambda\in(0,\lambda')$, $\lambda'$ being sufficiently small, where $\mathcal{V}$ is a constant and $2<\kappa<r^-<{2_s^*}^-$.
\end{theorem}
\begin{proof}
	We first choose those sequences that will be sufficient to prove this lemma. We shall discard it if $\{u_i\}\subset X_0$ is an eventually zero sequence.  Assume that $\{u_i\}\subset X_0$ is a sequence with infinitely many terms equal to $0$. Then we can choose  a subsequence of $\{u_i\}$ that is nonzero. Let this sequence $\{u_i\}$ be such that
	\begin{align}\label{ps 1}
		\bar{I}(u_i)\rightarrow c~\text{and}~ \Lambda_{\bar{I}}(u_i)\rightarrow 0
	\end{align}
	as $i\rightarrow\infty$. Let $\{v_i\}\subset Z'$ be such that $\Lambda_{\bar{I}}(u_i)=\|v_i\|_{Z'}$, where $Z'$ is the dual space of $Z$. By Remark \ref{SUBDIFF}, there exists $\eta_i\in\partial\bar{I}_{\eta}(u_i)$ satisfying the following
	\begin{align*}
		\begin{split}
		\langle v_i,\varphi\rangle=&a\iint_{Q}\frac{(u_(x)-u_i(y))}{|x-y|^{N+2s(x,y)}}(\varphi(x)-\varphi(y))dxdy+b\int_{\Omega}\nabla u_i\cdot\nabla \varphi dx\\
		& -\lambda\int_{\Omega}\phi(x,u_i)\varphi dx-\int_{\Omega}\left(\int_{\Omega}\frac{F(y,u_i(y))}{|x-y|^{\mu(x,y)}}dy\right)f(x,u_i)\varphi dx-\langle\eta_i,\varphi\rangle.
		\end{split}
	\end{align*}
	For $2<\kappa<r^-<{2_s^*}^-$ and $\kappa<\theta$, consider
	\begin{align*}
		&\bar{I}(u_i)-\frac{1}{\kappa}\langle u_i,u_i \rangle=
		 \left(\frac{1}{2}-\frac{1}{\kappa}\right)\|u_i\|^2+\frac{\lambda}{\kappa}\int_{\Omega}\phi(x,u_i)u_idx\nonumber\\
		&-\lambda\int_{\Omega}\Phi(x,u_i)dx-\frac{1}{2}\int_{\Omega}\left(\int_{\Omega}\frac{F(y,u_i(y))}{|x-y|^{\mu(x,y)}}dy\right)F(x,u_i)dx \nonumber\\
		 &+\frac{1}{\kappa}\int_{\Omega}\left(\int_{\Omega}\frac{F(y,u_i(y))}{|x-y|^{\mu(x,y)}}dy\right)f(x,u_i)u_idx-\eta\int_{\Omega}J_{\alpha}(u_i)dx+\frac{\eta}{\kappa}\langle\eta_i,u_i\rangle
		\nonumber
		\end{align*}
	\begin{align*}
		 \geq &
		 \left(\frac{1}{2}-\frac{1}{\kappa}\right)\|u_i\|^2+\lambda\left(\frac{1}{\kappa}-\frac{1}{1-\gamma^-}\right)\int_{\Omega}\psi(x,u_i)dx\nonumber\\
		 &+\left(\frac{\theta}{2\kappa}-\frac{1}{2}\right)\int_{\Omega}\left(\int_{\Omega}\frac{F(y,u_i(y))}{|x-y|^{\mu(x,y)}}dy\right)F(x,u_i)dx\nonumber\\
		 &+\eta\left(\frac{1}{\kappa}-\frac{1}{{r}^-}\right)\int_{\Omega\cap\{u_i>\alpha\}}|u_i|^{r(x)}dx+\frac{\eta}{\kappa}\langle\eta_i,u_i\rangle+\frac{\eta}{r^+}\alpha^{\mathcal{V}}|\{u_i=\alpha\}|\\
		\geq&
		 \left(\frac{1}{2}-\frac{1}{\kappa}\right)\|u_i\|^2-\frac{\lambda}{1-\gamma^+}C(\Omega,N)\|u_i\|^{1-\delta}\nonumber\\&+\eta\left(\frac{1}{\kappa}-\frac{1}{{r}^-}\right)\int_{\Omega\cap\{u_i>\alpha\}}|u_i|^{r(x)}dx.\nonumber
	\end{align*}
	Here,
	\begin{align*}
		\psi(x,u_i)=&\begin{cases}
			|u_i|^{1-\gamma(x)}, & ~\text{if}~u_i>\underline{u}_{\lambda}\\
			u_i\underline{u}_{\lambda}^{-\gamma(x)}, &~\text{if}~u_i\leq\underline{u}_{\lambda}
		\end{cases}
	\end{align*}
	and
	\begin{equation}\label{R}\mathcal{V}=\begin{cases}
			{r}^-, & \|w\|_{r(\cdot)} < 1\\
			{r}^+, & \|w\|_{r(\cdot)} > 1.
		\end{cases}
	\end{equation}
Define
	\begin{equation}\label{delta}\delta=\begin{cases}
			\gamma^{+}, & \|w^{1-\gamma(x)}\|_{\frac{1}{1-\gamma(x)}} < 1\\
			\gamma^{-}, & \|w^{1-\gamma(x)}\|_{\frac{1}{1-\gamma(x)}}  > 1.
	\end{cases}\end{equation}

	Furthermore,
	\begin{align}\label{eq3}
		\bar{I}(u_i)-\frac{1}{\kappa}\langle v_i,u_i\rangle&\leq |\bar{I}(u_i)|+\frac{1}{\kappa}\|v_i\|_{Z'}\|u_i\|
		\leq c+1+\|u_i\|+o_i(1).
	\end{align}
	Finally, we have
	\begin{align}\label{eq4}
		\begin{split}
		c&+1+\|u_i\|+o_i(1)\geq \bar{I}(u_i)-\frac{1}{\kappa}\langle v_i,u_i\rangle\\
		&\geq \left(\frac{1}{2}-\frac{1}{\kappa}\right)\|u_i\|^2-\frac{\lambda}{1-\gamma^+}C(\Omega,N)\|u_i\|^{1-\delta}\\
		&+\eta\left(\frac{1}{\kappa}-\frac{1}{{r}^-}\right)\int_{\Omega}H(u_i-\alpha)|u_i|^{r(x)}dx.
		\end{split}
	\end{align}
	This implies that the boundedness of the sequence $\{u_i\} $ in $Z$. Thus up to a subsequence, still denoted by the notation of the sequence, the following holds
	\begin{align}\label{eq5}
u_i\rightharpoonup u~\text{in}~Z,~u_i(x)\rightarrow u(x)~\text{a.e. in}~\Omega,\eta_i\stackrel{\ast}{\rightharpoonup} D~\text{in}~L^{r(x)'}(\Omega)~ \text{as}~i\rightarrow\infty,
	\end{align}
	\begin{eqnarray}\label{eq5'}
		& &\text{and}~\|u_i-u\|^2\rightarrow L~ \text{as}~i\rightarrow\infty
	\end{eqnarray}
	where $r(x)'=\frac{r(x)}{r(x)-1}$ is the conjugate of $r(x)$. If $L=0$, then  $u_i\rightarrow u$ and we have nothing to prove. Therefore, let $L>0$. From \eqref{eq5} we have
	\begin{eqnarray*}
		\int_{\Omega}|u_i|^{r(x)-1}\chi_{\{u_i>\alpha-\frac{1}{n}\}}\varphi dx\rightarrow\int_{\Omega}|u|^{r(x)-1}\chi_{\{u>\alpha\}}\varphi dx~\text{for all}~\varphi\in L^{r(x)}(\Omega).\nonumber
	\end{eqnarray*}
	Here, $\chi$ denotes the {\it indicator} function. Invoking the Brezis-Lieb Lemma \cite{9}, we have
	\begin{align}\label{eq7}
		\|u_i\|^2&=\|u_i-u\|^2+\|u\|^2+o_i(1),
	\end{align}
	\begin{align}\label{eq7'}
		\int_{\Omega}|u_i\chi_{\{u_i>\alpha-\frac{1}{n}\}}|^{r(x)} dx=&\int_{\Omega}|u_i\chi_{\{u_i>\alpha-\frac{1}{n}\}}-u\chi_{\{u>\alpha\}}|^{r(x)} dx\nonumber\\
		&+\int_{\Omega}|u\chi_{\{u>\alpha\}}|^{r(x)}dx+o_i(1).
	\end{align}
	The embedding result $W^{s(\cdot),2}(\Omega)\hookrightarrow L^{q(\cdot)}(\Omega),$ proved in \cite[Theorem $3.1$]{JMAA}, guarantees that $u_i\rightarrow u$ in $L^{q(x)}(\Omega),$ where $N>2s(x,y)$ for $(x,y)\in\bar{\Omega}\times\bar{\Omega}$ and there exists $\varepsilon=\varepsilon(x)>0$ such that
	\begin{equation}\label{A2}
		\underset{y\in\Omega_{x,\epsilon}}{\sup}\{q(y)\}\leq\frac{2N}{N-2\underset{(y,z)\in\Omega_{x,\varepsilon}\times\Omega_{x,\epsilon}}{\inf} \{s(y,z)\}},
	\end{equation}
	where $\Omega_{z,\varepsilon}=B_{\varepsilon}(z)\bigcap\Omega,$ for  $z\in\Omega$.
	Also, the following always holds

	  $$\int_{\Omega}\phi(x,u_i)dx\leq\int_{\Omega}{u}_{i}^{-\gamma(x)}dx.
	  \ \
$$ 
	 	  Thus it follows that
	\begin{align}\label{eq8}
		\left|\int_{\Omega}\phi(x,u_i)(u_i-u)dx\right|&\leq o_i(1).
	\end{align}
	
	Consider now
	\begin{align*}
			o_i(1)=&\langle u_i,u_i-u\rangle
			=\\&a\iint_{Q}\frac{(u_i(x)-u_i(y))}{|x-y|^{N+2s(x,y)}}((u_i-u)(x)-(u_i-u)(y))dxdy\\
			&+b\int_{\Omega}\nabla u_i\cdot\nabla(u_i-u)dx-\lambda\int_{\Omega}\phi(x,u_i)(u_i-u) dx\\
			&-\int_{\Omega}\left(\int_{\Omega}\frac{F(y,u_i(y))}{|x-y|^{\mu(x,y)}}dy\right)f(x,u_i)(u_i-u) dx-\langle\eta_i,u_i-u\rangle\\
			=&\|u_i\|^2-\iint_{Q}\frac{(u_i(x)-u_i(y))}{|x-y|^{N+2s(x,y)}}(u(x)-u(y))dxdy\\
			&-\lambda\int_{\Omega}\phi(x,u_i)(u_i-u)dx\\
			&-\int_{\Omega}\left(\int_{\Omega}\frac{F(y,u_i(y))}{|x-y|^{\mu(x,y)}}dy\right)f(x,u_i)(u_i-u) dx-\langle\eta_i,u_i-u\rangle
			\end{align*}
		\begin{align*}
			=&\|u_i-u\|^2-\int_{\Omega}|u_i|^{r(x)-1}(u_i-u)\chi_{\{u_i>\alpha-\frac{1}{n}\}}dx+o_i(1)\\
			&(\text{invoking}~\eqref{eq5}, \eqref{eq7})\\
			=&\|u_i-u\|^2-\int_{\Omega}|u_i|^{r(x)}\chi_{\{u_i>\alpha-\frac{1}{n}\}}dx+\int_{\Omega}|u|^{r(x)}\chi_{\{u>\alpha\}}dx\\
			&+o_i(1)\\
			=&\|u_i-u\|^2-\int_{\Omega}|u_i\chi_{\{u_i>\alpha-\frac{1}{n}\}}-u\chi_{\{u>\alpha\}}|^{r(x)} dx+o_i(1)\\
			\geq&\|u_i-u\|^2-\|u_i\chi_{\{u_i>\alpha-\frac{1}{n}\}}-u\chi_{\{u>\alpha\}}\|_{r(\cdot)}^{\mathcal{V}} +o_i(1).
	\end{align*}
	Therefore

	\begin{eqnarray}\label{eq11}
		S_1&\leq \frac{\|u_i-u\|^2}{\|u_i\chi_{\{u_i>\alpha-\frac{1}{n}\}}-u\chi_{\{u>\alpha\}}\|_{r(\cdot)}^2}\leq L^{\frac{\mathcal{V}-2}{\mathcal{V}}}+o_i(1).
	\end{eqnarray}
	Thus $L>{S_1}^{\frac{\mathcal{V}}{\mathcal{V}-2}}+o_i(1)$. It now follows from \eqref{eq4}  that
	\begin{align*}
		c\geq &\left(\frac{1}{2}-\frac{1}{\kappa}\right)\|u_i\|^2-\frac{\lambda}{1-\gamma^+}C(\Omega,N)\|u_i\|^{1-\delta}\\
		\geq &\left(\frac{1}{2}-\frac{1}{\kappa}\right)(L+\|u\|^2)-\frac{\lambda}{1-\gamma^+}C(\Omega,N)\|u_i\|^{1-\delta}\nonumber\\
		\geq&\left(\frac{1}{2}-\frac{1}{\kappa}\right)(L+\|u\|^2)-\left(\frac{1}{2}-\frac{1}{\kappa}\right)\|u\|^2\nonumber\\
		 &-\left(\frac{1}{2}-\frac{1}{\kappa}\right)^{-\frac{1-\delta}{1+\delta}}\left(\frac{\lambda}{1-\gamma^+}C(\Omega,N)\right)^{\frac{2}{1+\delta}}\nonumber\\
		\geq & \left(\frac{1}{2}-\frac{1}{\kappa}\right){S_1}^{\frac{\mathcal{V}}{\mathcal{V}-2}}-\left(\frac{1}{2}-\frac{1}{\kappa}\right)^{-\frac{1-\delta}{1+\delta}}\left(\frac{\lambda}{1-\gamma^+}C(\Omega,N)\right)^{\frac{2}{1+\delta}}+o_i(1)\nonumber\\
		=&c^*+o_i(1)\nonumber
	\end{align*}
	with $c^*>0$ for a small range of $\lambda$, say $\lambda\in(0,\lambda')$. This leads to a contradiction to the assumption that $c<c^*$. Therefore $L=0$ and hence $u_i\rightarrow u$ in $Z$ as $i\rightarrow\infty$.
\end{proof}
  The next lemma will guarantee the verification of the hypotheses of the Mountain Pass Theorem for nondifferentiable functionals.
\begin{lemma}\label{hypo_of_MP}
	Let $(H_1)-(H_2)$ and $(F_1)-(F_2)$ hold. Then there exists $r_1>0$, $\lambda_0=\lambda_0(r_1)$, $r_2>0$,
	$m_0>0$, $\sigma\in Z$ such that for every $\alpha>0$ and $\lambda\in(0,\lambda_0)$ the following holds
	\begin{enumerate}
		\item $\underset{m\in[0,m_0]}{\sup}\bar{I}(m\sigma)<c^*$;
		\item $\bar{I}(v)\geq r_2$ for every $v\in \partial B_{r_1}(0)\cap Z$, where $r_1,r_2$ are independent of $\alpha$,and  $\partial B_{r_1}(0)=\{v\in Z:\|v\|=r_1\}$;
		\item $\bar{I}(m_0\sigma)<0$ for $\|m_0\sigma\|>r_1$.
	\end{enumerate}
\end{lemma}
\begin{proof}
	Let $\sigma\in Z$ be fixed with $\sigma>0$ in $\Omega$ and $\|\sigma\|=1$. Now for sufficiently small $m_0>0$ and $\lambda\in(0,\lambda'')$ for sufficiently small $\lambda''$, we have
	\begin{align}\label{eq13}
		\bar{I}(m\sigma)&\leq \frac{1}{2}m_0^2-\frac{\lambda}{1-\gamma^-}\int_{\Omega}m_0^{1-\gamma(x)}\sigma^{1-\gamma(x)}dx<0
		\ \ \text{and} \ \
		\frac{1}{2}m_0^2< c^*
	\end{align}
	for every $m\in[0,m_0]$, thus proving properties $1$ and $3$. Furthermore, by the embedding result in \cite[Theorem $3.1$]{JMAA},
	\begin{align}\label{ineq1}
		\bar{I}(u)&\geq\frac{1}{2}\|u\|^2-\lambda C_1\|u\|^{1-\delta}-C_2\|u\|^{2\mathcal{A}}-\eta C_3\|u\|^{\mathcal{V}}
	\end{align}
	where
	\begin{equation}\label{A}\mathcal{A}=\begin{cases}
			\tau^-, & \|v\|_{\beta^+\tau(\cdot)} < 1\\
			\tau^+, & \|v\|_{\beta^+\tau(\cdot)}> 1.
	\end{cases}\end{equation}
	Here, we have used the Hardy-Littlewood-Sobolev inequality for variable exponent given in \cite[Proposition $2.4$]{alv1} and the conditions in $(F_1)$, to estimate the Choquard term.
	Define
	$$\tilde{m}(t)=t^{1+\delta}-2C_2t^{2\mathcal{A}-1+\delta}-2\eta C_3t^{\mathcal{V}-1+\delta}-2\lambda C_1.$$
	This function achieves a maximum, say, at $t=r_1>0$, i.e. $\underset{t\in[0,1]}{\max}\{\tilde{m}(t)\}=\tilde{m}(r_1)>0$. We let $$\lambda'''=\frac{1}{2C_1}\tilde{m}(r_1)
	\
	\hbox{ and}
	\
	\lambda_0=\min\{\lambda',\lambda'',\lambda'''\},$$ where $\lambda'$ is obtained from  Theorem \ref{PS_THM}. Then for any $u \in Z$ such that $\|u\|=r_1$ and for any $\lambda\in(0,\lambda_0)$ we have $$\bar{I}(u)\geq \frac{1}{2}r_1^{1-\delta}\tilde{m}(r_1)=r_2.$$ Thus for $\lambda\in (0,\lambda_0),$ property $2$ holds.
\end{proof}
\section{Proofs of the main results}\label{proof}
\noindent This section is devoted to the proofs of the main results stated in Theorems \ref{mainres1} and \ref{mainres2}.

\subsection{Proof of Theorem \ref{mainres1}}
Let  $m_0$, $\sigma$, $r_1$, $r_2$, $\lambda'$ be as given in Lemma $\ref{hypo_of_MP}$. We further denote $$c_\alpha=\underset{\zeta\in \Gamma}{\inf}~\underset{t\in [0,1]}{\max}~\bar{I}(\zeta(t))~\text{and}~\Gamma=\{\zeta\in C([0,1];Z):\zeta(0)=0,\zeta(1)=m_0\sigma\}.$$ Since $\bar{I}$ fulfills the conditions of the theorem (as proved in Lemma $\ref{hypo_of_MP}$), the existence of a (PS)$_{c_{\alpha}}$ sequence $\{w_i\}\subset Z$, such that $|w_i|>\underline{u}_{\lambda}$ a.e. in $\Omega$, is guaranteed that obeys $\bar{I}(w_i)=c_\alpha+o_i(1)$ and $\Lambda_{\bar{I}}(w_i)=o_i(1)$. The properties given in Lemma $\ref{hypo_of_MP}$ imply
\begin{equation}\label{starrr 7}
	r_2\leq c_\alpha< c^*=\left(\frac{1}{2}-\frac{1}{\kappa}\right){S_1}^{\frac{\mathcal{V}}{\mathcal{V}-2}}-\left(\frac{1}{2}-\frac{1}{\kappa}\right)^{-\frac{1-\delta}{1+\delta}}\left(\frac{\lambda}{1-\gamma^+}C(\Omega,N)\right)^{\frac{2}{1+\delta}}~
\end{equation}
for all $\alpha>0$. Thus, by  \cite[Theorem $3.1$]{JMAA} there exists $w_\alpha\in Z$ such that, up to a subsequence which is still denoted by the notation of the sequence, $w_i\rightarrow w_\alpha$ in $Z$ as $n\rightarrow \infty$. This establishes $\bar{I}(w_\alpha)=c_\alpha$ and $0\in \partial \bar{I}(w_\alpha)$. Further, by \eqref{3.9} and Lemma $\ref{properties I}$, we have
\begin{align*}
	&a\iint_{Q}\frac{(w_\alpha(x)-w_\alpha(y))}{|x-y|^{N+2s(x,y)}}(\varphi(x)-\varphi(y))dxdy+b\int_{\Omega}\nabla w_{\alpha}\cdot\nabla\varphi dx\\
	&=\lambda\int_\Omega \Phi(x,w_\alpha)\varphi dx
	+\int_\Omega\left(\frac{F(y,u)}{|x-y|^{\mu(x,y)}}\right)dy\int_\Omega f(x,u)\varphi dx
	+\eta\int_\Omega j_\alpha\varphi dx,\nonumber
\end{align*}
for every $\varphi\in Z$, where $f_\alpha(x)\in [\underline{f}_\alpha(w_\alpha(x)),\overline{f}_\alpha(w_\alpha(x))]$ a.e. in $\Omega$. Clearly $|w_{\alpha}|\geq \underline{u}_{\lambda}>0$ a.e. in $\Omega$. Therefore, $w_{\alpha}$ is a nontrivial weak solution of \eqref{main3}. 

We now claim is that the Lebesgue measure of the set $\{x\in\Omega:w_{\alpha}(x)>\alpha\}$ is positive. Suppose to the contrary, that $|\{x\in\Omega:w_{\alpha}(x)>\alpha\}|=0$. Then we have $w_{\alpha}\leq\alpha$ a.e. in $\Omega$. Therefore, from the weak formulation of the problem \eqref{main3} and on testing with $\varphi=w_{\alpha}$ we get
\begin{align}\label{ineq2}
	 2r_2\leq\|w_{\alpha}\|^2=&\lambda\int_{\Omega}\phi(x,w_{\alpha})w_{\alpha}dx+\int_{\Omega}\int_{\Omega}\left(\frac{F(y,w_{\alpha})}{|x-y|^{\mu(x,y)}}dy\right)f(x,w_{\alpha})w_{\alpha}dx\nonumber\\
	&+\int_{\Omega}j_{\alpha}w_{\alpha}dx\nonumber\\
	\leq & \lambda\alpha^{1-\delta}|\Omega|+\frac{C}{\tau^-}|\Omega|^{\frac{1}{\tau^+}}\alpha^{2}+\alpha^{\mathcal{V}}|\Omega|.
\end{align}
This is impossible since $r_2$ is independent of $\alpha$. Thus there exists a small $\alpha$, say $\alpha'$, such that for all $\alpha\in (0,\alpha')$ we have $|\{x\in\Omega:w_{\alpha}(x)>\alpha\}|>0$.
\qed\\

Since by Theorem \ref{mainres1} we have proved the existence of a nontrivial solution to \eqref{main prob} it will now be interesting to investigate the nature of this obtained solution under the limit $\alpha\rightarrow 0^+$.

\begin{theorem}\label{limitingcase1}
	Let $r_2$ be as in Lemma \ref{hypo_of_MP} and $c_{\alpha}, c_0$ be as in the proof of Theorem \ref{mainres1}. Then $\underset{\alpha\rightarrow 0^+}{\lim}c_{\alpha}=c_0\geq r_2$.
\end{theorem}
\begin{proof}
	Let $\bar{I}_0:Z\rightarrow\mathbb{R}$ be the associated functional to the problem \eqref{main prob} for $\alpha=0$ which is defined as follows.
	\begin{align}\label{alphazerofunc}
		 \bar{I}_0(u)=&\frac{1}{2}\|u\|^2-\lambda\int_{\Omega}\frac{|u|^{1-\gamma(x)}}{1-\gamma(x)}dx-\int_{\Omega}\int_{\Omega}\left(\frac{F(y,u)}{|x-y|^{\mu(x,y)}}dy\right)F(x,u)dx\nonumber\\
		&-\int_{\Omega}H(u)\frac{|u|^{r(x)}}{r(x)}dx~\text{for all}~ u\in Z.
	\end{align}
	Define
	\begin{align}\label{zeroenergylevel}
		c_0&=\underset{\xi\in\Gamma}{\inf}\underset{t\in [0,1]}{\max}\bar{I}_0(\xi(t))
	\end{align}
	where $$\Gamma=\{\xi\in C([0,1];Z):\xi(0)=0,\xi(1)=m_0\sigma\}.$$ The notations $m_0, \sigma$ are as in Theorem \ref{hypo_of_MP}. Since $$J_{\alpha}(t)=\int_0^tH(t-\alpha)t^{r(x)-1}dt=\chi_{\{t>\alpha\}}\frac{1}{r(x)}(|t|^{r(x)}-\alpha^{r(x)}),$$ this implies that
	\begin{align}\label{starrr9'}
		\left|\int_{\Omega}\frac{1}{r(x)}|u|^{r(x)}dx\right.&\left.-\int_{\Omega}J_{\alpha}(u)dx\right| 
			 \nonumber\\
		&=\left|\int_{\Omega}\frac{|u|^{r(x)}}{r(x)}\chi_{\{u\leq\alpha\}}dx
		+\int_{\Omega}\frac{\alpha^{r(x)}}{r(x)}\chi_{\{u>\alpha\}}dx\right|
		\leq \frac{2\alpha^{\mathcal{V}}}{{r}^-}.
	\end{align}
	Clearly, $\bar{I}_0(u)\leq\bar{I}(u)$ for all $u\in Z$. Thus $c_0\leq c_{\alpha}$ for any $\alpha>0$. From \eqref{starrr9'} we establish that
	\begin{align}\label{starrr11'}
		\bar{I}(u)&=\bar{I}_0(u)+o_{\alpha}(1)~\text{for all}~u\in Z.
	\end{align}
	This further gives
	\begin{align}\label{starrr12'}
		\bar{I}(\xi(t))&=\bar{I}_0(\xi(t))+o_{\alpha}(1)~\text{for all}~\xi\in\Gamma,~t\in[0,1]
	\end{align}
	and thus $c_{\alpha}=c_0+o_{\alpha}(1)$. So from \eqref{starrr12'} and using the properties proved in Lemma \ref{hypo_of_MP} we conclude from the Mountain Pass Theorem that $\underset{\alpha\rightarrow 0^+}{\lim}c_{\alpha}=c_0$.
\end{proof}

\subsection{Proof of Theorem \ref{mainres2}}
Let $w_{\alpha}\in Z$ be a weak solution of \eqref{main prob}. So,
\begin{align}\label{alphaform}
	a\iint_{Q}&\frac{(w_{\alpha}(x)-w_{\alpha}(y))}{|x-y|^{N+2s(x,y)}}(\varphi(x)-\varphi(y))dxdy+b\int_{\Omega}\nabla w_{\alpha}\cdot\nabla\varphi dx\\
	=&\lambda\int_{\Omega}|w_{\alpha}|^{-\gamma(x)-1}w_{\alpha}\varphi dx
	+\int_{\Omega}\int_{\Omega}\left(\frac{F(y,w_{\alpha})}{|x-y|^{\mu(x,y)}}dy\right)f(x,w_{\alpha})\varphi dx
	\nonumber\\
	&+\eta\int_{\Omega}j_{\alpha}\varphi dx
\end{align}	
for all $\varphi\in Z$. Consider a $(PS)_c$ sequence, say $(w_i)$, such that $w_i\rightarrow w_{\alpha}$ for a subsequence in $Z$ as $i\rightarrow\infty$. Therefore the following holds.
\begin{align*}
	c_{\alpha}+1+\|w_i\|+o(1)&\geq \left(\frac{1}{2}-\frac{1}{\kappa}\right)\|w_{\alpha}\|^2-\lambda\left(\frac{1}{1-\gamma^+}-\frac{1}{\kappa}\right)C(\Omega,N)\|w_i\|^{1-\delta}.
\end{align*}
Thus,
\begin{align*}
	c_{\alpha}+1+\|w_{\alpha}\|&\geq \left(\frac{1}{2}-\frac{1}{\kappa}\right)\|w_{\alpha}\|^2-\lambda\left(\frac{1}{1-\gamma^+}-\frac{1}{\kappa}\right)C(\Omega,N)\|w_{\alpha}\|^{1-\delta}
\end{align*}
for all $\alpha>0$. This implies that $w_{\alpha}$ is uniformly bounded in $Z$. 

Let $\{\alpha_i\}$ be a sequence such that $\alpha_i\rightarrow 0^+$. Therefore, there exists up to a subsequence, $w_0$ such that $w_{\alpha_i}\rightharpoonup w_0~\text{in}~Z.$ From the embedding result in \cite[Theorem $3.1$]{JMAA} we have $w_{\alpha_i}\rightarrow w_{0}$ in $L^{q(x)}(\Omega).$
Hence, $w_{\alpha_i}(x)\rightarrow w_{0}(x)$ a.e. in $\Omega$. Also, $j_{\alpha_i}(w_{\alpha_i})\rightarrow |w_0|^{r(x)-1}w_0$. Thus
\begin{align*}
	&\iint_{Q}\frac{(w_0(x)-w_0(y))}{|x-y|^{N+2s(x,y)}}(\varphi(x)-\varphi(y))dxdy=
	\lambda\int_{\Omega}|w_0|^{-\gamma(x)-1}w_0\varphi dx\nonumber\\
	&+\int_{\Omega}\int_{\Omega}\left(\frac{F(y,w_0)}{|x-y|^{\mu(x,y)}}dy\right)f(x,w_0)\varphi dx
	+\eta\int_{\Omega}|w_0|^{r(x)-2}w_0\varphi dx~\text{for all}~\varphi\in Z.
\end{align*}
In the above equation, Remark \ref{remark limit} in Section~\ref{appendix} has been used to pass the limit in the singular term. Furthermore, using \eqref{starrr 7}, we obtain
\begin{align}\label{ineq6}
	r_2&\leq \bar{I}(w_{\alpha_i})=c_{\alpha_i}\nonumber\\
	 <c^*=&\left(\frac{1}{2}-\frac{1}{\kappa}\right){S_1}^{\frac{\mathcal{V}}{\mathcal{V}-2}}-\left(\frac{1}{2}-\frac{1}{\kappa}\right)^{-\frac{1-\delta}{1+\delta}}\left(\frac{\lambda}{1-\gamma^+}C(\Omega,N)\right)^{\frac{2}{1+\delta}}
\end{align}
for all $n\in\mathbb{N}$. From the inequality \eqref{starrr9'} we have
\begin{align}\label{eq14}
	c_{\alpha_i}=\bar{I}(u_{\alpha_i})=\bar{I}_0(u_{\alpha_i})+o_{\alpha_i}(1).
\end{align}
Thus by \eqref{starrr11'}-\eqref{ineq6} we see that
\begin{align}\label{starrr10}
	w_{\alpha_i}\rightarrow w_0,~\text{in}~Z~\text{as}~\alpha_i\rightarrow 0^+.
\end{align}
So $c_0=\bar{I}_0(w_0)\geq r_2>0$. This completes the proof.
\qed
\section{Appendix}\label{appendix}
 In this section we shall prove the existence of a solution to the {\it singular} problem.
\begin{theorem}\label{imp_thm}
	For any $\lambda\in (0,\lambda_0)$ there exists a positive weak solution $\underline{u}_{\lambda}$ to \begin{align}\label{sing_prob}
		\begin{split}
			a(-\Delta)^{s(\cdot)}u+b(-\Delta)u&=\lambda |u|^{-\gamma(x)-1}u,~\text{in}~\Omega\\
			u=0,~\text{in}~\mathbb{R}^N\setminus\Omega.
		\end{split}
	\end{align} such that $\underline{u}_{\lambda}<u$ a.e. in $\Omega,$ where $u$ is a solution to \eqref{main prob}.
\end{theorem}
\begin{proof}
	We first prove the existence of $\underline{u}_{\lambda}$. We observe that the functional $\bar{I}_{\text{sing}}$  associated to \eqref{sing_prob} can be defined as follows:
	\begin{align}
		\label{ef_auxZ}
		\begin{split}
			\bar{I}_{\text{sing}}(u)=&\frac{a}{2}\iint_{Q}\frac{|u(x)-u(y)|^2}{|x-y|^{N+2s(x,y)}}dxdy+\frac{b}{2}\int_{\Omega}|\nabla u|^2dx-\lambda\int_{\Omega}|u|^{1-\gamma(x)}dx
		\end{split}
	\end{align}
	for $u\in Z$. The rest of the proof follows the argument in the proof of   {\ Choudhuri-Repov\v{s}} \cite[Lemma A.2]{chou-rep}.
\end{proof}

\begin{remark}\label{remark limit}
	We claim that
	\begin{equation}\label{singular limit}
		\lim\limits_{\alpha_i\rightarrow 0^+}\int_\Omega\frac{\varphi}{v_{\alpha_i}^{\gamma(x)}}=\int_\Omega\frac{\varphi}{(v_{0}^+)^{\gamma(x)}}<\infty~\text{for all}~\varphi\in Z,
	\end{equation}
	where $v_{\alpha_i}$ is a positive weak solution to $\eqref{main prob}$ with $\alpha=\alpha_i$ as obtained in Theorem $\ref{mainres1}$ and $u_0$ is the uniform limit of the sequence $(u_{\alpha_i})$, up to a subsequence, in $\Omega$ obtained by \cite[Theorem $3.1$]{JMAA}.
	\begin{proof}[Proof of the claim]
		We first denote $$S_{\alpha_i}=\{x\in\Omega:v_{\alpha_i}(x)=0\}.$$ As $v_{\alpha_i}^{-\gamma(x)} \varphi\in L^1(\Omega)$ for every $\varphi\in Z$, the Lebesgue measure of $S_{\alpha_i}$ (denoted in short as $|S_{\alpha_i}|$) is zero, i.e. $|S_{\alpha_i}|=0$. Due to the sub-additive property of the Lebesgue measure implies, we have that $|\bigcup S_{\alpha_i}|=0$. 
		
		Let $|D|<\epsilon$ such that for $x\in\Omega\setminus D$, $v_0(x)=0$. We can construct the set $D$ by the Egorov Theorem. Therefore if we are able prove that 
		$$|T_{\alpha_i}|=|\{x\in\Omega\setminus D: v_{\alpha_i}(x)\rightarrow 0~\text{as}~{\alpha_i}\rightarrow0^+\}|=0,$$ then $\eqref{singular limit}$ gets proven. We define $$T_{\delta,{\alpha_i}}=\left\{x\in\Omega\setminus D:|v_{\alpha_i}(x)|<\delta\right\}.$$
		Apparently, for a fixed ${\alpha_i}$, by the uniform convergence, $|T_{\delta,{\alpha_i}}|\rightarrow 0$ as $\delta\rightarrow 0^+$. Furthermore, we have
		$$\underset{\delta,{\alpha_i}\in\mathbb{R}^+}{\bigcup}T_{\delta,{\alpha_i}}=\underset{\delta}{\bigcup}\underset{\delta\leq {\alpha_i}}{\bigcap}T_{\delta,{\alpha_i}}.$$
		Thus for a fixed ${\alpha_i}$, $$\left|\underset{\delta\leq {\alpha_i}}{\bigcap}T_{\delta,{\alpha_i}}\right|=\underset{\delta\rightarrow0^+}{\lim}T_{\delta,{\alpha_i}}=0.$$
		Hence, $$\left|\underset{\delta,{\alpha_i}\in\mathbb{R}^+}{\bigcup}T_{\delta,{\alpha_i}}\right|=0$$
		and so $|T_{\alpha_i}|=0$.
	\end{proof}
\end{remark}
\section*{Conclusions}
 We have analyzed the elliptic problem driven by a singularity, a Choquard term and a discontinuous power nonlinerity and have proved the existence of a solution. In the process we have derived an embedding result and have proved the existence of a positive global minimizer of the associated {\it singular problem}. We have also proved that as $\alpha\rightarrow 0^+$, the corresponding limiting solutions also converge to a solution of the above problem when $\alpha=0$. We also plan to extend the problem with a Kirchhoff type operator.
\section*{Acknowledgements}
We thank the referees for comments and suggestions.
\small{\section*{Declarations}
\subsection*{Funding}
Repov\v{s} was supported by the Slovenian Research Agency grants P1-0292, J1-4031, J1-4001, N1-0278, N1-0114, and N1-0083.
\subsection*{Conflicts of interest/Competing interests}
The authors declare to have no conflict of interest/competing interest.
\subsection*{Authors' contributions}
The authors contributed equally to this research.}

\bigskip 
\small 
\noindent

\begin{thebibliography} {999}

\bibitem{alv1} Alves, C.O., Tavares, L.S.: A Hardy-Littlewood-Sobolev-Type inequality for variable exponents and applications to quasilinear Choquard equations involving variable exponent. Mediterr. J. Math. \textbf{16},  27 pp (2019)

\bibitem{COA} Alves, C.O.,  Yuan, Z.,  Huang, L.: Existence and multiplicity of solutions for discontinuous elliptic problems in $\mathbb{R}^N$. Proc. R. Soc. Edinburgh Sect. A: Math. \textbf{151}, 548-572  (2021)

\bibitem{5} Ambrosetti, A., Calahorrano, M., Dobarro, F.: Global branching for discontinuous problems. Comment. Math. Univ. Carolin. \textbf{31}(2), 213-222 (1990)

\bibitem{arora1} Arora, R., R\u{a}dulesu, V.D., Combined effects in mixed local-nonlocal stationary problems, arXiv:2111.06701v1, (2021)

\bibitem{m4} Azroul, E.,  Benkirane, A., Shimi, M., Srati, M.: On a class of fractional $p(x)$-Kirchhoff type problems. Appl. Anal.	\textbf{100}, 383-402 (2019)

\bibitem{2022add} Biagi, S., Dipierro, S., Valdinoci, E., Vecchi, E.: A Brezis-Nirenberg type result for mixed local and nonlocal operators.   arXiv:2209.07502v1, (2022)

\bibitem{RBS} Biswas, R., Tiwari, S.: Variable order nonlocal Choquard problem with variable exponents. Complex Var. Elliptic Equ. \textbf{66}, 853-875 (2021)

\bibitem{d12}Blazevski, D., del-Castillo-Negrete, D.: Local and nonlocal anisotropic transport in reversed shear magnetic fields: shearless Cantori and nondiffusive transport. Phys. Rev. E. \textbf{87}, 063106 (2013)

\bibitem{9} Br\'{e}zis, H., Lieb, E.: A relation between pointwise convergence of functions and convergence of functionals. Proc. Amer. Math. Soc. \textbf{88}, 486-490 (1983)

\bibitem{DC} Cassani D., Vilasi, L., Wang, Y.: Local versus nonlocal elliptic equations: short-long range field interactions. Adv. Nonlinear Anal. \textbf{10}, 895-921 (2021)

\bibitem{crg} Chammem, R., Ghanmi, A., Sahbani, A.: Existence of solution for a singular fractional Laplacian problem with variable exponents and indefinite weights. Complex Var. Elliptic Equ. 66(8), 1320-1332 (2021)

\bibitem{16} Chang, K.C.: On the multiple solutions of the elliptic differential equations with discontinuous nonlinear terms. Sci. Sinica. \textbf{21}(2), 139-158 (1978)

\bibitem{17} Chang, K.C.: The obstacle problem and partial differential quations with discontinuous nonlinearilities. Comm. Pure Appl. Math. \textbf{33}(2), 117-146  (1980)

\bibitem{Chang 3} Chang, K.C.: Variational methods for non-differentiable functionals and their applications to partial differential equations. J. Math. Anal. Appl. \textbf{80}, 102-129 (1981)

\bibitem{dc_zamp} Choudhuri, D.: Existence and H\"{o}lder regularity of infinitely many solutions to a $p$-Kirchhoff type problem involving a singular nonlinearity without the Ambrosetti-Rabinowitz (AR) condition. Z. Angew. Math. Phys. \textbf{72}, 26 pp (2021)

\bibitem{chou-rep} Choudhuri, D., Repov\v{s}, D.D.: Elliptic problem driven by different types of
nonlinearities. Boundary Value Problems.  2021(85) (2021)

\bibitem{MGC}Crandall, M.G.,  Rabinowitz, P.H., Tartar, L.:  On a Dirichlet problem with a singular
nonlinearity. Commun. Partial. Differ. Equ. \textbf{2}, 193-222 (1977)

\bibitem{jcda} de Albuquerque, J.C., do. Santos, G.G.,  Figueiredo, G.M.: Existence and behavior of positive solutions for a class of linearly coupled systems with discontinuous nonlinearities in $\mathbb{R}^N$. J. Fixed Point Theory Appl. \textbf{23}, 27 pp (2021)

\bibitem{11} Di Nezza, E., Palatucci, G., Valdinoci, E.: Hitchhiker's guide to the fractional Sobolev spaces. Bull. Sci. Math. \textbf{136}, 521-573 (2012)	

\bibitem{d23}Dipierro, S., Proietti Lippi, E., Valdinoci, E.: (Non)local logistic equations with Neumann
conditions. Ann. Inst. H. Poincare Anal. Non
Lineaire (2022), DOI:10.4171/AIHPC/57

\bibitem{21} dos Santos, G.G., Figueiredo, G.M.: Solution for a Kirchhoff equation with critical Caffarelli-Kohn-Nirenberg growth and discontinuous nonlinearity. Z. Angew. Math. Phys. \textbf{69}, 13 pp (2018)

\bibitem{ggdos} dos Santos, G.G., Figueiredo, G.M., Nascimento, R.G.: Existence and behavior of positive solution for a problem with discontinuous nonlinearity in $\mathbb{R}^N$ via a nonsmooth penalization. Z. Angew. Math. Phys. \textbf{71}, 18 pp (2020)	

\bibitem{tava1} dos Santos, G.C.G., Tavares, L.S.: Existence and behavior of the solutions for an elliptic equation with a nonlocal operator involving critical and discontinuous nonlinearity. J. Math. Anal. Appl. \textbf{493}, Art ID: 124530, 17 pp (2021)

\bibitem{Fan2001}
Fan, X.L., Zhao, D.:  On the spaces $L^{p(x)}(\Omega)$ and $W^{m,p(x)}(\Omega)$.
J. Math. Anal. Appl. \textbf{263}, 424-446 (2001)

\bibitem{gmfrg}Figueiredo, G.M., Nascimento, R.G.: Existence of positive solutions for a class of $p\&q$ elliptic problem with critical exponent and discontinuous nonlinearity. Monatsh Math. \textbf{189}, 75-89 (2019)

\bibitem{GSC} Garain, P., Mukherjee, T.: Quasilinear nonlocal elliptic problems with variable singular exponent. Comm. on Pure and Applied Anal. \textbf{19}, 5059-5075 (2020)

\bibitem{MGV} Ghergu, M., R\v{a}dulescu, V.D.: Singular Elliptic Problems: Bifurcation and Asymptotic Analysis. Oxford: The Clarendon Press, Oxford University Press, (2008)

\bibitem{SGD} Ghosh, S. Choudhuri, D., Giri, R.K.: Infinitely many small solutions to an elliptic PDE of variable exponent with a singular nonlinearity. Complex Var. Elliptic Equ. \textbf{66}(11), 1797--1817 (2020)

\bibitem{26} Grossinho, M.R., Tersian, S.A.: An Introduction to Minimax Theorems and Their Applications to Differential Equations. Nonconvex Optimization and Its Applications, vol.52, Kluwer Academic Publishers, Dordrecht, (2001)

\bibitem{bh} Hu, B., Yang, Y.: A note on the combination between local and nonlocal $p$-Laplacian operators. Complex Var. Elliptic Equ. \textbf{65}, 1763-1776 (2020)

\bibitem{gdl} Li, G.D., Tang, C.L.: Existence of positive ground state solutions for Choquard equation with variable exponent growth. Discret. Contin. Dyn. Syst. Ser. S. \textbf{12}, 2035-2050 (2019)

\bibitem{vo} Lorenzo, C.F., Hartley, T.T.: Initialized fractional calculus. Int. J. Appl. Math. \textbf{3}, 249-266 (2000)

\bibitem{papa1} Papageorgiou, N.S., R\u{a}dulescu, V.D., Repov\v{s}, D.D.:  Nonlinear Analysis - Theory and Methods. Springer Monographs in Mathematics. Springer, Cham, 2019.

\bibitem{32} R\v{a}dulescu, V.D.: Mountain pass theorems for nondifferentiable functions and applications, Proc. Jpn. Acad., Ser. A, Math. Sci. \textbf{69}, 193-198 (1993)

\bibitem{RD}
R\v{a}dulescu, V.D., Repov\v{s}, D.D.: Partial Differential Equations with Variable Exponents: Variational Methods and Qualitative Analysis, CRC Press, Boca Raton, (2015)

\bibitem{KAD} Saoudi, K., Panda, A., Choudhuri, D.: A singular elliptic problem involving fractional $p$-Laplacian and a discontinuous critical nonlinearity. J. Math. Phys. \textbf{62}, Art ID: 071505, 15 pp (2021)

\bibitem{lz} Wang, L., Zhang, B.: Infinitely many solutions for Kirchhoff-type variable-order fractional Laplacian problems involving variable exponents.  Appl. Anal. 100(11), 2418-2435 (2021)

\bibitem{mdby} Xiang, M., Hu, D., Zhang, B., Wang, Y.: Multiplicity of solutions for variable-order fractional Kirchhoff equations with nonstandard growth. J. Math. Anal. Appl. \textbf{501}, 19 pp (2020)

\bibitem{MXBZ}  Xiang, M., Zhang, B.:  A critical fractional $p$-Kirchhoff type problem involving discontinuous nonlinearity. Discret. Contin. Dyn. Syst. Ser. S. \textbf{12}, 413-433 (2019)

\bibitem{nuovo} Xiang, M., Zhang, B., Yang, D.: Multiplicity results for variable-order fractional Laplacian equations with variable growth. Nonlinear Anal. \textbf{178}, 190-204 (2019)

\bibitem{ZYJX} Yuan, Z., Yu, J.: Existence of solutions for Dirichlet elliptic problems with discontinuous nonlinearity. Nonlinear Anal. 197, 17 pp, (2020)

\bibitem{ZJF} Zuo, J., Fiscella, A., Bahrouni, A.: Existence and multiplicity results for $p(.)\&q(.)$ fractional Choquard problems with variable order. Complex Var. Elliptic Equations. 67(2), 500-516 (2022)
\bibitem{JMAA} Zuo, J., Choudhuri, D., Repov\v{s}, D.D.: On critical variable-order Kirchhoff type problems with variable singular exponent. J. Math. Anal. Appl. \textbf{514}(1), Art ID: 126264, (2022)

\end{thebibliography}
\end{document}